\documentclass[a4paper,11pt]{article}
\usepackage[english]{babel}
\usepackage{url}
\usepackage{tikz}
\usepackage{tensor}
\usepackage{titling}
\usepackage[scaled=0.90]{helvet}
\usetikzlibrary{matrix,arrows,decorations.markings}
\usepackage[utf8]{inputenc}
\usepackage{latexsym}
\usepackage{graphicx} 
\usepackage{amsmath,amsfonts,amssymb,amsthm,url,textcomp}
\usepackage{amsfonts}
\usepackage{relsize}
\usepackage{mathabx}
\usepackage{mathtools}
\usepackage{enumerate}
\usepackage{hhline}
\usepackage{longtable}
\usepackage{listings}
\usepackage{csquotes}
\usepackage{a4wide}

\newtheorem{theorem}{Theorem}\numberwithin{theorem}{section}
\newtheorem{definition}[theorem]{Definition}
\newtheorem{lemma}[theorem]{Lemma}
\newtheorem{corollary}[theorem]{Corollary}
\newtheorem{proposition}[theorem]{Proposition}

\theoremstyle{plain}

\theoremstyle{remark}
\newtheorem{remark}[theorem]{Remark}

\begin{document}

\title{On finite groups where the order of every automorphism is a cycle length}

\author{Alexander Bors\thanks{The author is supported by the Austrian Science Fund (FWF):
Project F5504-N26, which is a part of the Special Research Program \enquote{Quasi-Monte Carlo Methods: Theory and Applications}. \newline 2010 \emph{Mathematics Subject Classification}: 15A21, 20D15, 20D25, 20D45, 20D60. \newline \emph{Key words and phrases:} Finite groups, automorphisms, cycle structure}}

\date{}

\maketitle

\begin{abstract}
Using Frobenius normal forms of matrices over finite fields as well as the Burnside Basis Theorem, we give a direct proof of Horo\v{s}evski\u{\i}'s result that every automorphism $\alpha$ of a finite nilpotent group has a cycle whose length coincides with $\mathrm{ord}(\alpha)$. Also, we give two new sufficient conditions for an automorphism $\alpha$ of an arbitrary finite group to satisfy this property, namely when $\mathrm{ord}(\alpha)$ is a product of at most two prime powers or when $\alpha$ has a sufficiently large cycle. This will allow us to show that the least order of a group where this property is violated for an appropriate automorphism is $120$. Finally, we observe that any finite group embeds both into a finite group with this property (as all finite symmetric groups enjoy the property) as well as into a finite group not having this property.
\end{abstract}

\section{Motivation and some terminology}

We denote by $\mathbb{N}$ the set of natural numbers (including $0$) and by $\mathbb{N}^+$ the set of positive integers. For any set $X$, $\mathcal{S}_X$ denotes the symmetric group on $X$, and for a subset $M$ of the domain of a function $f$, we denote by $f[M]$ the pointwise image of $M$ under $f$. As a motivation for the notion studied in this paper, we point out the following concept:

\begin{definition}\label{fdsDef}
A \textbf{finite dynamical system} (abbreviated henceforth by FDS) is a finite set $X$ together with an \textbf{endofunction of $X$}, i.e., a function $f:X\rightarrow X$.
\end{definition}

FDSs have gained a lot of research interest in recent years, which is partially due to their great importance for practical applications, ranging from cryptography and pseudorandom number generation (see, for instance, \cite{GOS14a} and \cite{OS10a}) to reverse engineering (\cite{JLSS07a}). Especially for pseudorandom number generation, one requires certain properties of a periodic FDS (an FDS $(X,f)$ where $f\in\mathcal{S}_X$), which correspond to distribution properties of the pseudorandom sequence generated from it, see \cite{TW07a}. One necessary condition for an FDS to be of practical use in this respect is that a large portion of the elements of $X$ lie on \enquote{long} cycles of the permutation $f$. Also, computation of values of $f$ should, of course, be efficient, which can be ensured by equipping the set $X$ with an appropriate algebraic structure with respect to which $f$ is defined, the most intensely studied case being where $X$ is a Cartesian power $k^n$ of a finite field $k$ and $f$ a rational map $k^n\rightarrow k^n$. We are interested in the case where $X$ is endowed with a finite group structure and $f$ is a permutation of $X$ respecting that group structure, i.e., a group automorphism. In view of what was said above, we want to better understand the possible cycle structures of automorphisms of finite groups. It is now time to introduce the notion discussed in this paper, partially following the terminology from the recent paper \cite{GPS15a}:

\begin{definition}\label{terminologyDef}
(1) Let $X$ be a finite set, $\sigma\in\mathcal{S}_X$. A cycle of $\sigma$ whose length coincides with the order of $\sigma$ is called a \textbf{regular cycle of $\sigma$}, and if $\sigma$ has a regular cycle, we say that $\sigma$ (or the periodic finite dynamical system $(X,\sigma)$) satisfies the \textbf{regular cycle condition} (RCC).

(2) A finite structure $A$ belonging to a class $\mathcal{C}$ of structures such that all automorphisms of $A$, viewed as permutations of the underlying set, satisfy the RCC (we also speak of \textbf{RCC-automorphisms}, as opposed to \textbf{non-RCC-automorphisms}) is called an \textbf{RCC-$\mathcal{C}$-struc-}\textbf{ture} (examples of this terminology are \enquote{RCC-group} or \enquote{RCC-ring}; we may also say that \textbf{$A$ satisfies the RCC}). A $\mathcal{C}$-structure which is not an RCC-$\mathcal{C}$-structure is called a \textbf{non-RCC-$\mathcal{C}$-structure}.
\end{definition}

It is not difficult to see that a permutation $\sigma$ of a finite set $X$ satisfies the RCC if and only if all its cycle lengths divide the largest among them. This is what happens for all automorphisms of the most elementary examples of finite groups, such as finite cyclic groups. However, the class of finite RCC-groups consists of much more than just these. The first paper discussing the RCC in finite groups known to the author is \cite{Hor74a}, where Horo\v{s}evski\u{\i} (who spoke of \enquote{faithful cycles} instead of \enquote{regular cycles}) proved sufficiency of each of the following conditions for the RCC of an automorphism $\alpha$ of a finite group $G$:

(1) $G$ is nilpotent (Corollary 1 in \cite{Hor74a}).

(2) $G$ has no nontrivial normal solvable subgroups (Theorem 1 in \cite{Hor74a}).

(3) $\mathrm{ord}(\alpha)$ and $|G|$ are coprime (Corollary 2 in  \cite{Hor74a}).

Furthermore, he gave a series of examples of finite supersolvable non-RCC-groups. The aim of this paper is to further elaborate on the RCC in finite groups. In Section 2, we will give an alternative proof of the RCC in finite nilpotent groups. Section 3 provides two new sufficient conditions on the pair $(G,\alpha)$ for $\alpha$ to satisfy the RCC. As an application, we will prove in Section 4 that the least order of a finite non-RCC-group is $120$. We conclude by observing in Section 5 that every finite group embeds both into an RCC-group as well as into a finite non-RCC-group.

\section{On the RCC in finite nilpotent groups}

Recall that for a group $G$ and an automorphism $\alpha$ of $G$, a subgroup $H\leq G$ is called \textit{$\alpha$-admissible} if and only if $\alpha[H]=H$, and that if $N$ is an $\alpha$-admissible normal subgroup of $G$, then $\alpha$ induces a unique automorphism $\tilde{\alpha}$ on the quotient $G/N$ such that the following diagram commutes:

\begin{center}
\begin{tikzpicture}
\matrix (m) [matrix of math nodes, row sep=3em,
column sep=3em]
{ G & G \\
G/N & G/N \\
};
\path[->]
(m-1-1) edge node[above] {$\alpha$} (m-1-2)
(m-1-1) edge node[left] {$\pi$} (m-2-1)
(m-1-2) edge node[right] {$\pi$} (m-2-2)
(m-2-1) edge node[below] {$\tilde{\alpha}$} (m-2-2);
\end{tikzpicture}
\end{center}

For a characteristic subgroup $N$ of $G$, the function $\mathrm{Aut}(G)\rightarrow\mathrm{Aut}(G/F)$, mapping $\alpha$ to $\tilde{\alpha}$ as above, is a group homomorphism, whose kernel is denoted by $\mathrm{Aut}_N(G)$.

Horo\v{s}evski\u{\i}'s proof of the RCC in nilpotent groups is a consequence of the following, which is Theorem 3 in his paper \cite{Hor74a} and gives information on the action of non-RCC automorphisms in minimal examples. 

\begin{theorem}\label{horTheo1}
Let $\alpha$ be a non-RCC-automorphism of a finite group $G$ inducing RCC-automorphisms on every proper $\alpha$-admissible subgroup and on every quotient of $G$ by an $\alpha$-admissible subgroup. Also, assume that all proper powers of $\alpha$ satisfy the RCC. Then $\alpha$ acts identically on every $\alpha$-admissible nilpotent normal subgroup $H$ of $G$.\qed
\end{theorem}

The proof that every finite nilpotent group is an RCC-group then is by contradiction: If $G$ is a counterexample of minimal order, then we can let $H:=G$ in Theorem \ref{horTheo1} and get a contradiction. Our proof that all finite nilpotent groups are RCC-groups also follows from a stronger result, and it is direct. We begin by observing that it suffices to show the RCC for finite $p$-groups by point (2) of the following lemma:

\begin{lemma}\label{splitLem}
(1) If a finite group $G$ has a non-RCC direct factor, then $G$ is non-RCC.

(2) If $G_1,\ldots,G_s$ are finite RCC-groups such that for all $i,j\in\{1,\ldots,n\}$ with $i\not=j$, we have $\mathrm{Hom}(G_i,G_j)=0$, then $G_1\times\cdots\times G_s$ is an RCC-group.
\end{lemma}

\begin{proof}
For (1): Let $G=H\times K$, where $H$ does not satisfy the RCC. Fix an automorphism $\alpha_1$ of $H$ which does not satisfy the RCC. Then the set of cycle lengths of the automorphism $\alpha_1\times\mathrm{id}_K$ of $G$ equals the set of cycle lengths of $\alpha_1$, whence $\alpha_1\times\mathrm{id}_K$ does not satisfy the RCC.

For (2): Let $\alpha$ be an automorphism of $G_1\times\cdots\times G_s$. By assumption, $\alpha$ decomposes as a product $\alpha_1\times\cdots\times\alpha_s$ of automorphisms of the single $G_i$. For $i=1,\ldots,s$, let $L_i$ denote the largest cycle length of $\alpha_i$. Then for any point $(g_1,\ldots,g_s)\in G_1\times\cdots\times G_s$, denoting by $l_i$ the cycle length of $g_i$ under $\alpha_i$ (which is a divisor of $L_i$), we find that the cycle length of $(g_1,\ldots,g_s)$ under $\alpha$ is equal to $\mathrm{lcm}\{l_i\mid i=1,\ldots,s\}$, which is a divisor of $\mathrm{lcm}\{L_i\mid i=1,\ldots,s\}$. On the other hand, if $h_i\in G_i$ is chosen such that its cycle length under $\alpha_i$ equals $L_i$, then the cycle length of $(h_1,\ldots,h_s)$ under $\alpha$ equals $\mathrm{lcm}\{L_i\mid i=1,\ldots,s\}$. Hence $\alpha$ satisfies the RCC.
\end{proof}

\begin{remark}\label{splitRem}
Point (1) of Lemma \ref{splitLem} implies that if there exists any finite non-RCC-group, then there even exist infinitely many. For if $G_0$ is a finite non-RCC-group and $G$ is any finite group, then $G_0\times G$ is also non-RCC.
\end{remark}

We will also need the following easy observation from group-theoretic dynamics:

\begin{proposition}\label{subgroupProp1}
Let $G$ be any group (n.n.~finite), $\varphi$ an endomorphism of $G$ and $e\in\mathbb{N}^+$. Then $\mathrm{per}_e(\alpha):=\{g\in G\mid \varphi^e(g)=g\}$ (the set of points in $G$ which are periodic under $\varphi$ with period a divisor of $e$) is a subgroup of $G$, and for all $e_1,e_2\in\mathbb{N}^+$, if $e_1\mid e_2$, then $\mathrm{per}_{e_1}(\varphi)\leq\mathrm{per}_{e_2}(\varphi)$. In particular, if $G=\langle g_1,\ldots,g_r\rangle$ and there exist $e_1,\ldots,e_r\in\mathbb{N}^+$ such that $\varphi^{e_i}(g_i)=g_i$ for $i=1,\ldots,r$, then $\varphi$ is an automorphism of $G$ whose order is finite and a divisor of $\mathrm{lcm}\{e_1,\ldots,e_r\}$ (and equal to $\mathrm{lcm}\{e_1,\ldots,e_r\}$ if each $e_i$ is minimal).\qed
\end{proposition}

In view of this, the RCC in finite $p$-groups immediately follows from the following theorem:

\begin{theorem}\label{pTheo}
Let $G$ be a group of order $p^m$ such that $|G/\mathrm{Frat}(G)|=p^r$. Then for any automorphism $\alpha$ of $G$, setting $\mathrm{ford}(\alpha):=\mathrm{ord}(\tilde{\alpha})$, where $\tilde{\alpha}$ is the image of $\alpha$ under the canonical homomorphism $\mathrm{Aut}(G)\rightarrow\mathrm{Aut}(G/\mathrm{Frat}(G))$, there exist $x_1,\ldots,x_r\in G$ and $k_1,\ldots,k_r\in\{0,\ldots,(m-r)r\}$ such that $G=\langle x_1,\ldots,x_r\rangle$ and for $i=1,\ldots,r$, the cycle length under $\alpha$ of $x_i$ equals $p^{k_i}\cdot\mathrm{ford}(\alpha)$.
\end{theorem}

The proof of Theorem \ref{pTheo} is split into two parts: We first treat the case where $G$ is elementary abelian. The \enquote{jump} from that special case to the general case is not as big as it may seem, since the Frattini subgroup of a finite $p$-group establishes a close connection between automorphisms of that $p$-group and automorphisms of some finite elementary abelian $p$-group. We repeat the according well-known results as soon as we need them.

It turns out that Frobenius normal forms are a powerful tool for studying automorphisms of finite elementary abelian groups; we shall briefly repeat the basic theory. Recall that, as a consequence of the structure theorem for finitely generated modules over a principal ideal domain, for any field $K$ (not necessarily algebraically closed), any $n\in\mathbb{N}^+$ and any $A\in\mathrm{Mat}_{n,n}(K)$, there exists a matrix $U\in\mathrm{GL}_n(K)$ such that $U^{-1}AU$ is a matrix in Frobenius normal form, i.e., it is a block diagonal matrix the blocks of which each are of the form $$F=\begin{pmatrix}0 & 0 & \cdots & 0 & -a_0 \\ 1 & 0 & \cdots & 0 & -a_1 \\ \vdots & \vdots & \vdots & \vdots & \vdots \\ 0 & 0 & \cdots & 1 & -a_{r-1}\end{pmatrix}.$$ For such a Frobenius block matrix $F$, the monic polynomial $p_F(X)=a_0+a_1X+\cdots+a_{n-1}X^{n-1}+X^n\in K[X]$ is its characteristic polynomial (in particular, the matrix is regular if and only if $a_0\not=0$), and the Frobenius block matrix is called the \textit{companion matrix of $p_F(X)$}.

Sticking with the above notation, if $p_1(X),\ldots,p_s(X)$ are the polynomials of which the Frobenius blocks of the Frobenius normal form of $A$ are the companion matrices (possibly with repetitions), then there exists an isomorphism $K^n\rightarrow \prod_{i=1}^s{K[X]/(p_i(X))}$ of $K$-vector spaces under which the action of $A$ corresponds to the multiplication with the element $$(X+(p_1(X)),\ldots,X+(p_s(X)))$$ in the $K$-algebra $$\prod_{i=1}^s{K[X]/(p_i(X))}.$$ This is useful because it allows us to answer certain questions on automorphisms of $(\mathbb{Z}/p\mathbb{Z})^n$ via a regress to the theory of finite fields.

For example, it is immediate by this theory that all finite elementary abelian groups are RCC-groups: If $\alpha$ is any automorphism of $(\mathbb{Z}/p\mathbb{Z})^n$, then with respect to an appropriate $\mathbb{F}_p$-basis, $\alpha$ is represented by a matrix in Frobenius normal form. But then $\alpha$ corresponds to the product map of the multiplications with the images of $X$ under the canonical projections $\mathbb{F}_p[X]\rightarrow\mathbb{F}_p[X]/(p_i(X))$ on the various invariant subspaces on which the diagonal blocks of the matrix act. Now on the one hand, the cycle length of $$(1+(p_1(X)),\ldots,1+(p_s(X)))$$ under this multiplication obviously is the least common multiple of the orders of the images $X+(p_i(X))$ in the various quotient algebras (which are units, as the constant terms of the $p_i(X)$ are nonvanishing), and on the other hand, denoting this least common multiple by $m$, for any element $$(r_1(X)+(p_1(X)),\ldots,r_s(X)+(p_s(X)))$$ from the product of the quotient algebras, its image under the $m$-fold iteration of multiplication with $$(X+(p_1(X)),\ldots,X+(p_s(X)))$$ is $$(X^mr_1(X)+(p_1(X)),\ldots,X^mr_s(X)+(p_s(X)))=(r_1(X)+(p_1(X)),\ldots,r_s(X)+(p_s(X))),$$ proving that $\alpha$ satisfies the RCC.

The following theorem is precisely the statement of Theorem \ref{pTheo} for elementary abelian $p$-groups:

\begin{theorem}\label{elAbTheo}
Let $\alpha$ be an automorphism of $(\mathbb{Z}/p\mathbb{Z})^n$. Then there exists an $\mathbb{F}_p$-basis $v_1,\ldots,v_n$ of $(\mathbb{Z}/p\mathbb{Z})^n$ such that the cycle lengths under $\alpha$ of all the $v_i$ are equal to the order of $\alpha$.
\end{theorem}

This was found independently by Giudici, Praeger and Spiga in \cite{GPS15a} (see Lemma 4.1 there). Their proof has some basic ideas in common with ours (such as directly decomposing the entire vector space into subspaces invariant and indecomposable under the action of $\alpha$, which actually is just the decomposition associated with the block decomposition of the Frobenius normal form of $\alpha$), but our technique of using Frobenius normal forms and, associated with them, polynomials over finite fields, is different from theirs. We give our proof here in order to illustrate the usefulness of this concept.

\begin{proof}[Proof of Theorem \ref{elAbTheo}.]
We begin by observing that it suffices to find a generating set of $(\mathbb{Z}/p\mathbb{Z})^n$ all of whose elements lie on cycles of length $\mathrm{ord}(\alpha)$. Choose an $\mathbb{F}_p$-basis $B=B_1\cup\cdots\cup B_s$ such that $\alpha$ with respect to $B$ is represented by a matrix in Frobenius normal form, where the $B_i$ are bases for the invariant subspaces corresponding to the various diagonal blocks, and let $p_i(X)$ denote the polynomial for which the $i$-th diagonal block is the companion matrix. We identify, under an appropriate isomorphism as mentioned above, $(\mathbb{Z}/p\mathbb{Z})^n$ with $\prod_{i=1}^s{K[X]/(p_i(X))}$. If we denote by $S_i$ the set of points in the $i$-th factor $K[X]/(p_i(X))$ which lie on a cycle of maximal length, then $S_1\times\cdots\times S_s$ is a subset of the set $S$ of points of the entire product algebra whose cycle length under $\alpha$ coincides with $\mathrm{ord}(\alpha)$. Now it is not difficult to see that on the one hand, for each $i=1,\ldots,s$, the \enquote{canonical basis vectors} $1+(p_i(X)),X+(p_i(X)),\ldots,X^{d_i-1}+(p_i(X))$ all are elements of $S_i$ (whence the case $s=1$ is clear and we may assume henceforth that $s\geq 2$), and on the other hand, all differences of elements of $S_i$ lie in the span of $S$ (by considering the various differences of two elements of $S_1\times\cdots\times S_s$ where all but the $i$-th components are equal). If $p$ is odd, then we additionally find that $S_i$ is closed under scalar multiplication with $2$ (because this is a unit in $\mathbb{F}_p$) and hence that the span of $S$ contains, for every $i$, a basis of the $i$-th invariant subspace, whence it is a generating subset, as desired.

So we are left to treat the case $p=2$, in which the above considerations at least yield that all differences $1+X+(p_i(X)),1+X^2+(p_i(X)),\ldots,1+X^{d_i-1}+(p_i(X))$ lie in the span of $S$. Now it is not difficult to see that these form a basis for the hyperplane of elements that are sums of an even number of the canonical basis vectors, so we would be done if we could show that for each $i=1,\ldots,s$, the span of $S$ contains an element of $S_i$ which is a sum of an odd number of the basis vectors. Actually, it would suffice if we could show this for all but one $i$, since then, we know that all $S_j$ for $j\not=i$ are subsets of the span of $S$, and hence $S_i$ as well by appropriate subtractions. We now need some important observations from the theory of finite fields:

(1) If we factor $p_i(X)=\prod_{k=1}^{s_k}{q_{i,k}(X)^{e_{i,k}}}$, where the $q_{i,k}$ are pairwise distinct irreducible polynomials and we assume w.l.o.g.~that $q_{i,1}(X)$ always is $1+X$, possibly with exponent $e_{i,1}=0$, then the order of $X$ in the quotient algebra $\mathbb{F}_2[X]/(p_i(X))$, which is also known as the \textit{order of $p_i(X)$}, see \cite{LN97a}, pp.~84ff., can (by an application of the Chinese Remainder Theorem) be computed as the least common multiple of the orders of the $q_{i,k}(X)^{e_{i,k}}$, and the order of a power of an irreducible polynomial $q_{i,k}(X)^{e_{i,k}}$ is equal to $2^{\lceil\mathrm{log}_2(e_{i,k})\rceil}\cdot\mathrm{ord}(\alpha)$, where $\alpha$ is any root of $q_{i,k}$ in an appropriate splitting field of $q_{i,k}$ over $\mathbb{F}_2$.

(2) Call a polynomial $q(X)\in\mathbb{F}_2[X]$ \textit{even} if it is sum of an even number of monomials, and \textit{odd} otherwise. Associated with this attribution of a parity to polynomials is a surjective ring homomorphism $\pi:\mathbb{F}_2[X]\rightarrow\mathbb{Z}/2\mathbb{Z}$.

(3) The only even irreducible polynomial over $\mathbb{F}_2$ is $1+X$, since every even polynomial has $1$ as a root.

Note that for a point of the product algebra, the property of lying on a cycle of length $\mathrm{ord}(\alpha)$ is equivalent to the least common multiple of the cycle lengths of its components being equal to $\mathrm{ord}(\alpha)$. Fix $i$ such that the $2$-adic valuation of the maximal cycle length from the $i$-th component coincides with the $2$-adic valuation of $\mathrm{ord}(\alpha)$. Then for any $j\not=i$, we claim that we can obtain an element from $S_j$ which is the sum of an odd number of basis vectors in the span of $S$ as follows: If $p_j(X)$ is odd, then $X^{\mathrm{deg}(p_j(X))}+(p_j(X))$ is an element in $S_j$ which is a sum of an even number of basis vectors, so the difference of it with $1+(p_j(X))$ is an element in the span of $S$ which is the sum of an odd number of basis vectors. And if $p_j(X)$ is even, then for $l=2,\ldots,s_j$, let $\mathcal{U}_l:=\{q(X)\in\mathbb{F}_2[x]\mid \mathrm{deg}(q(X))<e_{j,l}\cdot\mathrm{deg}(q_{j,l}(X))\hspace{3pt}\text{and}\hspace{3pt}\mathrm{gcd}(q(X),p_l(X))=1\}$ and let $\iota_l:\mathcal{U}_l\rightarrow\mathcal{U}_l$ denote the function that assigns to each representative from $\mathcal{U}_l$ its unique multiplicative inverse modulo $q_{j,l}(X)^{e_{j,l}}$. Now consider the following two elements of the product algebra: $$(1+(p_1(X)),\ldots,1+(p_s(X)))$$ and $$(1+(p_1(X)),\ldots,1+(p_{j-1}(X)),$$$$\sum_{l=2}^{s}{\frac{p_j(X)}{q_{j,l}(X)^{e_{j,l}}}\iota_l(\frac{p_j(X)}{q_{j,l}(X)^{e_{j,l}}})}+(p_j(X)),1+(p_{j+1}(X)),\ldots,1+(p_s(X))).$$ The $j$-th entry of the second tuple corresponds, under the canonical isomorphism $$\mathbb{F}_2[X]/(p_j(X))\rightarrow\prod_{l=1}^{s_j}{\mathbb{F}_2[X]/(q_{j,l}^{e_{j,l}}}),$$ to $$((1+X)^{e_{j,1}},1+(q_{j,2}^{e_{j,2}}),\ldots,1+(q_{j,s_j}^{e_{j,s_j}})).$$ Now since by (1), a factor $(1+X)^e$ in the factorization of a polynomial over $\mathbb{F}_2$ only contributes a power of $2$ to the order, the cycle length contributed by the $j$-th component of the second tuple has all $p$-adic valuations equal to the ones of the maximal cycle length in that component \textit{except} possibly for the $2$-adic valuation, which, however, is taken care of in the $i$-th component. We conclude that still, the cycle length of the point represented by the second tuple is $\mathrm{ord}(\alpha)$. Taking the difference of the two tuples, we obtain that $$1+\sum_{l=2}^{s_j}{\frac{p_j(X)}{q_{j,l}(X)^{e_{j,l}}}\iota_l(\frac{p_j(X)}{q_{j,l}(X)^{e_{j,l}}})}+(p_j(X))$$ is an element from the $j$-th component contained in the span of $S$. However, since $\frac{p_j(X)}{q_{j,l}(X)}$ is even for $l=2,\ldots,s_j$, we obtain that said element of the $j$-th component can be written as a sum of an odd number of basis vectors, and we are done.
\end{proof}

\begin{corollary}\label{vectorCor}
All finite vector spaces satisfy the RCC.
\end{corollary}

\begin{proof}
Just observe that any finite vector space has a finite elementary abelian group as a reduct.
\end{proof}

We now extend the result to all finite $p$-groups. Recall that the Frattini subgroup of a group $G$, denoted by $\mathrm{Frat}(G)$, is defined to be the intersection of all the maximal subgroups of $G$ (which is understood to equal $G$ if $G$ has no maximal subgroups). Clearly, $\mathrm{Frat}(G)$ is a characteristic subgroup of $G$, and the following are well-known results on the Frattini subgroup in finite $p$-groups (to be found, for instance, in \cite[p.~140]{Rob96a}):

\begin{theorem}[The Burnside Basis Theorem, \cite{Bur14a}]\label{burnsideTheo}
Let $G$ be a finite $p$-group. Then $\mathrm{Frat}(G)=G'G^p$, that is, $\mathrm{Frat}(G)$ is the smallest normal subgroup $N$ of $G$ such that the quotient $G/N$ is elementary abelian. Furthermore, if $r=\mathrm{dim}_{\mathbb{F}_p}(G/\mathrm{Frat}(G))$ and $x_1,\ldots,x_r\in G$ are such that $(x_1\mathrm{Frat}(G),\ldots,x_r\mathrm{Frat}(G))$ is an $\mathbb{F}_p$-basis for $G/\mathrm{Frat}(G)$, then $G=\langle x_1,\ldots,x_r\rangle$.\qed
\end{theorem}

\begin{theorem}[P.~Hall, \cite{Hal33a}]\label{hallTheo}
If $G$ is a group of order $p^m$ such that $|G/\mathrm{Frat}(G)|=p^r$, then the order of $\mathrm{Aut}_{\mathrm{Frat}(G)}(G)$ divides $p^{(m-r)r}$ and the order of $\mathrm{Aut}(G)$ divides $|\mathrm{GL}_r(p)|\cdot p^{(m-r)r}$.\qed
\end{theorem}

Using these results, we are now ready to prove Theorem \ref{pTheo}.

\begin{proof}[Proof of Theorem \ref{pTheo}.]
Choose, by Theorem \ref{elAbTheo}, an $\mathbb{F}_p$-basis $x_1\mathrm{Frat}(G),\ldots,x_r\mathrm{Frat}(G)$ for $G/\mathrm{Frat}(G)$ such that for $i=1,\ldots,r$, the cycle length under $\tilde{\alpha}$ of $x_i\mathrm{Frat}(G)$ coincides with $\mathrm{ford}(\alpha)$. Then by commutativity of the diagram

\begin{center}
\begin{tikzpicture}
\matrix (m) [matrix of math nodes, row sep=3em,
column sep=3em]
{ G & G \\
G/\mathrm{Frat}(G) & G/\mathrm{Frat}(G) \\
};
\path[->]
(m-1-1) edge node[above] {$\alpha$} (m-1-2)
(m-1-1) edge node[left] {$\pi$} (m-2-1)
(m-1-2) edge node[right] {$\pi$} (m-2-2)
(m-2-1) edge node[below] {$\tilde{\alpha}$} (m-2-2);
\end{tikzpicture}
\end{center}

we obtain that the cycle lengths of the $x_i$ under $\alpha$ are all divisible by $\mathrm{ford}(\alpha)$. On the other hand, they all divide $\mathrm{ord}(\alpha)$, which by Theorem \ref{hallTheo} is a divisor of $p^{(m-r)r}\cdot\mathrm{ford}(\alpha)$, and the result follows.
\end{proof}

\begin{corollary}\label{pCor1}
Any finite nilpotent group is an RCC-group.\qed
\end{corollary}

\begin{corollary}\label{pCor2}
All finite rings are RCC-rings.
\end{corollary}

\begin{proof}
This follows immediately from Corollary \ref{pCor1}, by observing that any finite ring has a finite abelian group as a reduct.
\end{proof}

\section{Two conditions sufficient for the RCC}

Since the order of a permutation of a finite set is the least common multiple of its cycle lengths, it is clear that a permutation of a finite set whose order is a power of a prime has a regular cycle, whereas for every composite number $o$, there exists an $n\in\mathbb{N}^+$ and a $\sigma\in\mathcal{S}_n$ of order $o$ without a regular cycle. In the first part of this section, we will show that an automorphism $\alpha$ of a finite group $G$ satisfies the RCC if its order is a product of at most two prime powers, whereas for every other natural number $o$, there exists a non-RCC-automorphism of a finite group whose order equals $o$.

The proof for sufficiency of said condition builds up on the following observation, which can be seen as a strengthening of Lemma 4 from \cite{Hor74a}:

\begin{lemma}\label{dominanceLem}
Let $G$ be any group and let $\alpha$ be an automorphism of $G$ of finite order $o$. If $x,y\in G$ such that for some prime $p$, denoting by $l_x$ and $l_y$ the cycle length under $\alpha$ of $x$ and $y$ respectively, $\mathrm{max}\{\nu_p(l_x),\nu_p(l_y)\}>\mathrm{min}\{\nu_p(l_x),\nu_p(l_y)\}$, then the cycle length $l$ under $\alpha$ of $xy$ satisfies $\nu_p(l)=\mathrm{max}\{\nu_p(l_x),\nu_p(l_y)\}$.
\end{lemma}

\begin{proof}
Observing that the cycle lengths under $\alpha$ of $g$ and $g^{-1}$ are the same for all $g\in G$, we may w.l.o.g.~assume that $k_1:=\nu_p(l_x)>\nu_p(l_y)=:k_1'$. If $p=p_1,p_2,\ldots,p_r$ is a finite list of primes containing all the prime divisors of $\mathrm{lcm}(l_x,l_y)$, then setting, for $i=2,\ldots,r$, $K_i:=\nu_{p_i}(\mathrm{lcm}(l_x,l_y))$, by Proposition \ref{subgroupProp1}, $xy$ is an element of $\mathrm{per}_{p_1^{k_1}p_2^{K_2}\cdots p_r^{K_r}}(\alpha)$, whence in particular, $\nu_p(l)\leq k_1$. If $\nu_p(l)<k_1$, then $xy$ would be an element of $$\mathrm{per}_{p_1^{k_1-1}p_2^{K_2}\cdots p_r^{K_r}}(\alpha).$$
However, since $$\alpha^{p_1^{k_1-1}p_2^{K_2}\cdots p_r^{K_r}}(x)\not=x,$$ we conclude that also $$\alpha^{p_1^{k_1-1}p_2^{K_2}\cdots p_r^{K_r}}(xy)=\alpha^{p_1^{k_1-1}p_2^{K_2}\cdots p_r^{K_r}}(x)\cdot y\not=xy.$$
\end{proof}

As an immediate consequence, we obtain:

\begin{theorem}\label{dominanceTheo}
Let $G$ be any group and let $\alpha$ be an automorphism of $G$ of finite order $o$. Then the following hold:

(1) For any distinct primes $p,q$, $\alpha$ has a cycle whose length is divisible by $p^{\nu_p(o)}q^{\nu_q(o)}$.

(2) If the order of $\alpha$ is divisible by at most two distinct primes, then $\alpha$ satisfies the RCC.
\end{theorem}

\begin{proof}
For (1): Otherwise, since the order of $\alpha$ is the least common multiple of its cycle lengths, $\alpha$ would still have a cycle whose length is divisible by $p^{\nu_p(o)}$ (but not by $q^{\nu_q(o)}$) and a cycle whose length is divisible by $q^{\nu_q(o)}$ (but not by $p^{\nu_p(o)}$). Let $x$ be any point from the support of the first cycle, and let $y$ be a point from the support of the second cycle. Then by Lemma \ref{dominanceLem}, the cycle length under $\alpha$ of $xy$ is divisible by $p^{\nu_p(o)}q^{\nu_q(o)}$, a contradiction.

For (2): This follows immediately from (1).
\end{proof}

We end this first part of the section as promised, by showing the following:

\begin{proposition}\label{manyPrimeProp}
Let $o\in\mathbb{N}^+$ be divisible by at least three distinct primes. Then there exists a finite supersolvable group $G$ and a non-RCC-automorphism $\alpha$ of $G$ such that $\mathrm{ord}(\alpha)=o$.
\end{proposition}

We begin by treating the case $r=3$ separately, showing the following stronger statement in generalization of an example given in \cite{Hor74a}:

\begin{lemma}\label{threePrimeLem}
Define a multiplicative number-theoretic function $f:\mathbb{N}^+\rightarrow\mathbb{N}^+$ by $f(2^n)=2^{n+1}$ and, for $p>2$ prime, $f(p^n)=p^n$, for all $n\in\mathbb{N}$. Then if $p_1<p_2<p_3$ are primes and $(k_1,k_2,k_3)\in(\mathbb{N}^+)^3$, setting $o:=p_1^{k_1}p_2^{k_2}p_3^{k_3}$ there exists a finite supersolvable group $G_o$ of order $4f(o)$ having a non-RCC automorphism $\alpha$ of order $o$.
\end{lemma}

\begin{proof}
Define, for $i=1,2,3$, the group $B_i$ as $\mathbb{Z}/f(p_i^{k_i})\mathbb{Z}$, and let $B:=B_1\times B_2\times B_3$. The automorphism group of $B$ contains, for $i=1,2,3$, an element $\alpha_i$ acting identically on $B_i$ and inverting the elements from the other two factors. We have $\alpha_3=\alpha_1\alpha_2$, and the subgroup of $\mathrm{Aut}(B)$ generated by $\alpha_1$ and $\alpha_2$ is isomorphic to the Klein four group. Consider the natural semidirect product $G_o$ of $B$ with this automorphism group (a subgroup of $\mathrm{Hol}(B)$). It is clearly a supersolvable group of order $4f(o)$. Now consider the inner automorphism $\alpha$ of $G_o$ given by conjugation with the element $b_1b_2b_3$, where $b_i$, for $i=1,2,3$, is any generator of $B_i$. $\alpha$ acts identically on $B$, and it is not difficult to see that the cycle length under $\alpha$ of any element from the coset $\alpha_iB$, $i\in\{1,2,3\}$, equals $\frac{o}{p_i^{k_i}}$, whence $\mathrm{ord}(\alpha)=o$, and $\alpha$ has no regular cycle.
\end{proof}

\begin{remark}\label{threePrimeRem}
(1) Horo\v{s}evski\u{\i} in \cite{Hor74a} gave the construction for $o=3\cdot 5\cdot 7$.

(2) The least order of a non-RCC-group obtainable by this construction is $4\cdot 4\cdot 3\cdot 5=240$. Horo\v{s}evski\u{\i} also gave other examples of finite non-RCC-groups at the end of \cite{Hor74a}, but these are of even larger order.
\end{remark}

\begin{proof}[Proof of Proposition \ref{manyPrimeProp}.]
Let $o=p_1^{k_1}\cdots p_r^{k_r}$ be the prime factorization of $o$, with $r\geq 3$ and w.l.o.g.~$p_1<p_2<\cdots<p_r$. Fix a group $G_{p_1^{k_1}p_2^{k_2}p_3^{k_3}}$ as in Lemma \ref{threePrimeLem} and set $$G:=G_{p_1^{k_1}p_2^{k_2}p_3^{k_3}}\times\prod_{i=4}^r{\mathbb{Z}/p_i^{k_i+1}\mathbb{Z}}.$$ As a finite direct product of supersolvable groups, $G$ is supersolvable. Let $\alpha$ be a non-RCC-automorphism of the first factor of order $p_1^{k_1}p_2^{k_2}p_3^{k_3}$ and let, for $i=4,\ldots,r$, $\alpha_i$ be an automorphism of $\mathbb{Z}/p_i^{k_i+1}\mathbb{Z}$ of order $p_i^{k_i}$. Then it is readily checked that the product $\alpha\times\alpha_4\times\cdots\alpha_r$ is a non-RCC-automorphism of $G$ of order $o$.
\end{proof}

As for the second part of this section, we will show that an automorphism of a finite group satisfying some kind of \enquote{large cycle condition} also has a regular cycle. This builds up on other results on \enquote{large cycle automorphisms} from \cite{Bor15a}; for the reader's convenience, we quickly present those results that we need here.

\begin{definition}\label{lambdaDef}
Let $G$ be a finite group. For an automorphism $\alpha$ of $G$, define $\lambda(\alpha)$ to be the quotient of the largest cycle length of $\alpha$ by $|G|$. Also, define $\lambda(G)$ to be the maximum of the $\lambda(\alpha)$, where $\alpha$ runs through all automorphisms of $G$.
\end{definition}

\begin{theorem}\label{largeCycTheo1}(Theorem 1.7 in \cite{Bor15a}.)
Let $G$ be a finite group such that $\lambda(G)>\frac{1}{2}$. Then $G$ is abelian.\qed
\end{theorem}

More specifically, it is shown that if $(G,\alpha)$ is such that $G$ is a finite group and $\alpha$ is an automorphism of $G$ with $\lambda(G)>\frac{1}{2}$, then one of the following cases occurs:

(1) $G$ is an elementary abelian $2$-group and there exists a direct decomposition $G=\prod_{i=1}^r{H_i}$ such that $\alpha=\prod_{i=1}^r{\alpha_i}$ for automorphisms $\alpha_i$ of $H_i$ such that $\alpha_i$ permutes all nontrivial elements of $H_i$ in one cycle.

(2) $G$ is a primary cyclic $p$-group $\mathbb{Z}/p^k\mathbb{Z}$ for some odd prime $p$ and $\alpha$ is multiplication with a primitive root modulo $p^k$.

(3) $G$ is an elementary abelian $p$-group $(\mathbb{Z}/p\mathbb{Z})^n$ for some odd prime $p$ and either $n\not=2$ and $\alpha$ permutes all nontrivial elements of $G$ in one cycle, or $n=2$ and $\alpha$ again permutes all nontrivial elements in one cycle or $\alpha$ is given, with respect to an appropriate $\mathbb{F}_p$-basis, by a matrix of the form $A=\begin{pmatrix}-1 & -g^2 \\ 1 & 2g-1\end{pmatrix}$, where $g$ is a generator of $\mathbb{F}_p^{\ast}$.

(4) $G$ is a product of an elementary abelian $2$-group $G_2$ with either a primary cyclic $p$-group or an elementary abelian $p$-group $G_p$ for some odd prime $p$, and $\alpha$ decomposes as a product $\alpha_2\times\alpha_p$ of automorphisms over the two factors, both with a cycle filling more than half of the respective factor as well.

From this classification, we can deduce the following (which was not mentioned in \cite{Bor15a}):

\begin{corollary}\label{fixedPointFreeCor}
Let $G$ be a finite group and let $\alpha$ be an automorphism of $G$ such that $\lambda(\alpha)>\frac{1}{2}$. Then $\alpha$ is fixed-point free.
\end{corollary}

\begin{proof}
This is proved by verifying the assertion in each of the four cases listed above. (1) and (2) are clear. As for the exceptional case in (3), just observe that $\mathrm{det}(A-I)=(g-1)^2\not=0$ in $\mathbb{F}_p$. Finally, the assertion in case (4) follows from the other three cases.
\end{proof}

Another concept introduced in \cite{Bor15a} which we need here is the following:

\begin{definition}\label{affineDef}
Let $G$ be any group, $\varphi$ an endomorphism of $G$ and $g_0\in G$ fixed. The \textbf{(left) affine map of $G$ w.r.t.~$\varphi$ and $g_0$} is the function $\mathrm{A}_{\varphi,g_0}:G\rightarrow G$ mapping $g\mapsto g_0\cdot\varphi(g)$ for $g\in G$.
\end{definition}

The following is a slightly stronger version of Lemma 4.8 in \cite{Bor15a}:

\begin{lemma}\label{affineLem}
Let $G$ be a finite abelian group, let $\alpha$ be an automorphism of $G$ and let $g\in G$. Also, let $o_1$ denote the order of $\alpha$ and let $o_2$ denote the maximum order of a fixed point of $\alpha$. Then $\mathrm{ord}(\mathrm{A}_{\alpha,g})\mid o_1\cdot o_2$.
\end{lemma}

\begin{proof}
We show that the cycle length of any $a\in A$ under $\mathrm{A}_{\alpha,g}$ divides $o_1\cdot o_2$. It is not difficult to show by induction on $n\in\mathbb{N}$ that $\mathrm{A}_{\alpha,g}^n(a)=\alpha^n(a)\alpha^{n-1}(g)\cdots\alpha(g)g$. Now certainly $\alpha^{o_1\cdot o_2}(a)=a$, so it is equivalent to show that $\alpha^{o_1o_2-1}(g)\cdots\alpha(g)g=1$. This follows from the fact that the LHS is the product, for $k=o_2,\ldots,1$, of the group elements $$\alpha^{ko_1-1}(g)\cdots\alpha^{(k-1)o_1+}(g)\alpha^{(k-1)o_1}(g)=\alpha^{o_1-1}(g)\cdots\alpha(g)g,$$ so setting $x:=\alpha^{o_1-1}(g)\cdots\alpha(g)g$, it is equal to $x^{o_2}$. But by abelianity of $G$, it is readily checked that $x$ is a fixed point of $\alpha$, and we are done.
\end{proof}

We are now ready to show the following:

\begin{theorem}\label{largeCycRCCTheo}
Let $G$ be a finite group and $\alpha$ an automorphism of $G$ such that $\lambda(\alpha)\geq\frac{1}{3}$. Then $\alpha$ has a regular cycle.
\end{theorem}

\begin{proof}
Set $l:=\lambda(\alpha)\cdot|G|$. We need to show that $\mathrm{per}_l(\alpha)=G$. Now $\mathrm{per}_l(\alpha)$ contains all points on any cycle of $\alpha$ of length $l$ as well as the identity element of $G$, and hence $|\mathrm{per}_l(\alpha)|>\frac{1}{3}|G|$. So by Lagrange, the only case left to exclude is $[G:\mathrm{per}_l(\alpha)]=2$. In this case, $\alpha$ restricts to an automorphism $\tilde{\alpha}$ of $\mathrm{per}_l(\alpha)$ with $\lambda(\tilde{\alpha})\geq\frac{2}{3}$, so all cycle lengths of $\tilde{\alpha}$ divide $l$, and by Theorem \ref{largeCycTheo1}, $\mathrm{per}_l(\alpha)$ is abelian. Also, as observed in \cite{Bor15a}, after fixing a representative $x\in G\setminus\mathrm{per}_l(\alpha)$ for the coset $\mathrm{per}_l(\alpha)x$, the action of $\alpha$ corresponds, under the induced identification of elements from $\mathrm{per}_l(\alpha)x$ with elements from $\mathrm{per}_l(\alpha)$, to the action of an affine map $\mathrm{A}_{\tilde{\alpha},g}$ on $\mathrm{per}_l(\alpha)$ for an appropriate $g\in\mathrm{per}_l(\alpha)$. Since by Corollary \ref{fixedPointFreeCor}, $\tilde{\alpha}$ is fixed-point free, by Lemma \ref{affineLem}, all cycle lengths of $\alpha$ on $G\setminus\mathrm{per}_l(\alpha)$ divide $\mathrm{ord}(\tilde{\alpha})=l$, contradicting the proper inclusion of $\mathrm{per}_l(\alpha)$ in $G$.
\end{proof}

\section{The least order of a counterexample}

We shall now show that the smallest group order for which there exist examples of non-RCC groups is 120. Let us denote the number of points whose cycle length under $\alpha$ is precisely $d$ by $\zeta_d(\alpha)$ (so that $\frac{\zeta_d(\alpha)}{d}$ is the number of $d$-cycles of $\alpha$). Using Theorem \ref{dominanceTheo}, we find that if a finite group $G$ has a non-RCC automorphism $\alpha$ whose order is of the form $pqr$ for distinct primes $p,q,r$, the cycle structure of $\alpha$ is completely determined and we can also obtain some information on the structure of $G$:

\begin{lemma}\label{easiestCaseLem}
Let $G$ be a finite group such that for some pairwise distinct primes $p,q,r$, $G$ has an automorphism $\alpha$ which does not satisfy the RCC and whose order equals $pqr$. Then $|G|=4\cdot\zeta_1(\alpha)$, $\zeta_1(\alpha)=\zeta_{pq}(\alpha)=\zeta_{pr}(\alpha)=\zeta_{qr}(\alpha)$, and $\mathrm{fix}(\alpha):=\mathrm{per}_1(\alpha)\unlhd G$.
\end{lemma}

\begin{proof}
By Theorem \ref{dominanceTheo} and the assumption, $$\zeta_{pq}(\alpha),\zeta_{pr}(\alpha),\zeta_{qr}(\alpha)>0,\zeta_{pqr}(\alpha)=0.$$ $\alpha$ cannot have any $p$-cycles or $q$-cycles or $r$-cycles since otherwise, by Lemma \ref{dominanceLem}, we could obtain points on $pqr$-cycles. Fix any $x$ on a $pq$-cycle of $\alpha$. By Lemma \ref{dominanceLem} and what we already know about the cycle structure of $\alpha$, left multiplication with $x$ must map points on $pr$-cycles to points on $qr$-cycles and vice versa, and no point which does not lie on a $pr$- or a $qr$-cycle of $\alpha$ (and hence lies in $\mathrm{per}_{pq}(\alpha)$) is mapped to such a point. Since left multiplication with $x$ is a permutation of $G$, we conclude that $\zeta_{pr}(\alpha)=\zeta_{qr}(\alpha)$. Similary, one shows $\zeta_{pq}(\alpha)=\zeta_{qr}(\alpha)$, so we get that $$\zeta_{pq}(\alpha)=\zeta_{pr}(\alpha)=\zeta_{qr}(\alpha)=:\zeta.$$ Since $$|\mathrm{per}_{pq}(\alpha)|=\zeta_1(\alpha)+\zeta$$ and $$|G|=\zeta_1(\alpha)+3\zeta,$$ we conclude that $$\zeta_1(\alpha)+3\zeta=2\cdot(\zeta_1(\alpha)+\zeta),$$ or $$\zeta_1=\zeta,$$ from which $$|G|=4\zeta$$ follows. Also, if $y$ is any element of $G$ and $f$ is a fixed point under $\alpha$, it is now not difficult to see that $yxy^{-1}$ is also a fixed point: This holds in general if $y$ is a fixed point, so assume w.l.o.g.~that the cycle length under $\alpha$ of $y$ is $pq$. Now left and right multiplication with any point on a cycle of length $pq$ restricts to a permutation of $\mathrm{per}_{pq}(\alpha)$ that maps fixed points to points of cycle length $pq$; since the number of fixed points coincides with the number of points of cycle length $pq$, these multiplications therefore also map points of cycle length $pq$ to fixed points, and we are done.
\end{proof}

We are now ready to show:

\begin{theorem}\label{counterExTheo}
Let $G$ be a finite non-RCC group such that $|G|\leq 120$. Then $|G|=120$, and any non-RCC automorphism $\alpha$ of $G$ has order $30$ and satisfies $\zeta_1(\alpha)=\zeta_{6}(\alpha)=\zeta_{10}(\alpha)=\zeta_{15}(\alpha)=30$. Furthermore, $G$ is supersolvable.
\end{theorem}

\begin{proof}
It suffices to show that under the assumption $|G|\leq 120$, any non-RCC-automorphism of $G$ has order $30$; the rest follows from Lemma \ref{easiestCaseLem} (as well as the fact that groups of order $30$, as all groups whose Sylow subgroups are all abelian, are metacyclic, see Theorem 10.1.10 in \cite{Rob96a}). For this, in turn, it suffices to show that the order of any non-RCC automorphism of $G$ is of the form $p_1p_2p_3$ for primes $p_1<p_2<p_3$, since $\zeta:=\zeta_{p_1p_2}>0$ must, as the number of points on cycles of length $p_1p_2$, be divisible by $p_1p_2$, but it must also, as the number of points on cycles of length $p_1p_3$, be divisible by $p_1p_3$, whence it is a multiple of $p_1p_2p_3$, and $|G|\geq 4\cdot p_1p_2p_3$, which can only work out if $p_1=2,p_2=3,p_3=5$.

So let $\alpha$ be a non-RCC-automorphism of $G$. Note that by Theorem \ref{largeCycRCCTheo}, we obtain a contradiction is soon as we can derive that $\alpha$ has a cycle of length at least $40$. We first show that the order of $\alpha$ is divisible by precisely three distinct primes. By Theorem \ref{dominanceTheo}(2), it must be divisible by at least three distinct primes. If it was divisible by at least five distinct primes, say $p_1,\ldots,p_n$ in increasing order, then by Theorem \ref{dominanceTheo}(1), $\alpha$ would have a cycle of length bounded below by $p_{n-1}p_n\geq 7\cdot 11=77$, a contradiction. And if it is divisible by precisely four distinct primes, say $p_1<p_2<p_3<p_4$, it follows immediately that $p_1=2,p_2=3,p_3=5,p_4=7$ (since otherwise, we would again obtain a cycle of too large length) and $\mathrm{ord}(\alpha)=2^{k_1}\cdot3\cdot5\cdot7$, with $k_1\in\{1,2\}$. Now by Theorem \ref{dominanceTheo}(2), $\alpha$ has a cycle of length divisible by $35$. If it is additionally divisible by $2$ or $3$, it would be too large, so $\alpha$ actually has a cycle of length precisely $35$. But $\alpha$ also has a cycle of length divisible by $6$, and that cycle cannot have length equal to $6$, since otherwise, $\alpha$ would have a regular cycle by Lemma \ref{dominanceLem}. Hence there either is a cycle of length $30$ or $42$. The latter case is immediately contradictory, but in the first case, by Lemma \ref{dominanceLem}, multiplying a point on a cycle of length $35$ with a point with cycle length $30$ also yields a point whose cycle length is divisible by $2\cdot3\cdot7=42$, a contradiction. This proves that $\mathrm{ord}(\alpha)$ is divisible by precisely three distinct primes.

So say $\mathrm{ord}(\alpha)=p_1^{k_1}p_2^{k_2}p_3^{k_3}$ for primes $p_1<p_2<p_3$. Again in view of Theorem \ref{largeCycRCCTheo}, we conclude that $k_2=k_3=1$ and $k_1\in\{1,2\}$. It remains to exclude the case $k_1=2$. In that case, $\zeta_{p_1^2p_2}(\alpha)\zeta_{p_1^2p_3}(\alpha)>0$, $\zeta_{p_2p_3}(\alpha)+\zeta_{p_1p_2p_3}(\alpha)>0$ and all other potential cycle lengths of $\alpha$ divide $p_1p_2p_3$. Hence just as in the proof of Lemma \ref{easiestCaseLem}, fixing a point on a cycle of length $p_2p_3$ or $p_1p_2p_3$ and considering its left multiplication, we find that $\zeta_{p_1^2p_2}(\alpha)=\zeta_{p_1^2p_3}(\alpha)=:\zeta$ and can conclude that $\zeta\geq p_1^2p_2p_3\geq 2^2\cdot 3\cdot 5=60$, the final contradiction for this proof.
\end{proof}

On the other hand, we carried out a brute-force search with GAP \cite{GAP4} for non-RCC groups of order 120. The search revealed that of the $44$ isomorphism types of nonabelian groups of order $120$, precisely $8$ are non-RCC. The second parts of their GAP IDs are the numbers $8,9,10,11,12,13,14$ and $42$. We shall now work out an explicit proof that $\mathrm{SmallGroup}(120,8)$ is a counter-example:

\begin{proposition}\label{counterExProp}
The group $G:=\mathrm{SmallGroup}(120,8)$ does not satisfy the RCC.
\end{proposition}

\begin{proof}
We analyze $G$ from a presentation of it stored and outputted by GAP: $$G=\langle F_1,F_2,F_3,F_4,F_5\mid F_1^2=1,[F_1,F_2]=1,[F_1,F_3]=1,[F_1,F_4]=1,F_1F_5F_1^{-1}=F_5^4,$$$$F_3=F_2^2,[F_2,F_3]=1,F_2^{-1}F_4F_2=F_4^2,[F_2,F_5]=1,F_3^2=1,[F_3,F_4]=1,[F_3,F_5]=1,$$$$F_4^3=1,[F_4,F_5]=1,F_5^5=1\rangle.$$ As the generator $F_3$ is made superfluous by the relation $F_3=F_2^2$, we can remove it in the course of a Tietze transformation and obtain the following more concise presentation of $G$ (where we also omitted some (substituted) relations that follow from others): $$G=\langle F_1,F_2,F_4,F_5\mid F_1^2=1,[F_1,F_2]=1,[F_1,F_4]=1,F_1F_5F_1^{-1}=F_5^{-1},F_2^{-1}F_4F_2=F_4^{-1},$$$$[F_2,F_5]=1,F_2^4=1,F_4^3=1,[F_4,F_5]=1,F_5^5=1\rangle.$$ Now consider the following group presentations: $$\langle F_5\mid F_5^5=1\rangle$$ of $\mathbb{Z}/5\mathbb{Z}$, $$\langle F_4\mid F_4^3=1\rangle$$ of $\mathbb{Z}/3\mathbb{Z}$ and $$\langle F_1,F_2\mid F_1^2=1,F_2^4=1,[F_1,F_2]=1\rangle$$ of $\mathbb{Z}/2\mathbb{Z}\times\mathbb{Z}/4\mathbb{Z}$. We can obtain the second presentation of $G$ from these three in two steps: First, form the semidirect product of the second with the third of the introduced presentations by taking their disjoint union and adding the conjugation relations $[F_1,F_4]=1$ and $F_2^{-1}F_4F_2=F_4^{-1}$; this gives a presentation of a semidirect product $\mathbb{Z}/3\mathbb{Z}\rtimes(\mathbb{Z}/2\mathbb{Z}\times\mathbb{Z}/4\mathbb{Z})$. Secondly, form the semidirect product of the first introduced presentation with the just formed semidirect product presentation by taking their disjoint union and adding the conjugation relations $F_1F_5F_1^{-1}=F_5^{-1},[F_2,F_5]=1$ and $[F_4,F_5]=1$. This proves that $G$ is the semidirect product $\mathbb{Z}/5\mathbb{Z}\rtimes(\mathbb{Z}/3\mathbb{Z}\rtimes(\mathbb{Z}/2\mathbb{Z}\times\mathbb{Z}/4\mathbb{Z}))$, where, renaming the generators of the four canonical cyclic subgroups in the order as they appear in that notation by $x_1,x_2,x_3,x_4$, we have the conjugation relations $[x_1,x_2]=[x_1,x_4]=[x_2,x_3]=1,x_3x_1x_3^{-1}=x_1^{-1},x_4x_2x_4^{-1}=x_2^{-1}$.

By the normal form theorem for semidirect products, we can view the underlying set of the group $\mathbb{Z}/5\mathbb{Z}\times\mathbb{Z}/3\mathbb{Z}\times\mathbb{Z}/2\mathbb{Z}\times\mathbb{Z}/4\mathbb{Z}$ also as an underlying set for $G$, where the group operation $\cdot$ of $G$ is given by $$(k_1,k_2,k_3,k_4)\cdot(l_1,l_2,l_3,l_4)=(k_1+(-1)^{k_3}l_1,k_2+(-1)^{k_4}l_2,k_3+l_3,k_4+l_4).$$

Now consider the map (also found by a brute-force search with GAP) $\{x_1,x_2,x_3,x_4\}\rightarrow G$ given by $$x_1\mapsto x_1,x_2\mapsto x_2,x_3\mapsto x_1x_3x_4^2,x_4\mapsto x_2x_4^3.$$ It is readily checked that it respects the defining relations and hence extends to an endomorphism $\alpha$ of $G$. Also, one easily verifies the following: $$\alpha(k_1,k_2,k_3,k_4)=\begin{cases}(k_1+k_3,k_2,k_3,-k_4+2k_3) & \mbox{if }2\mid k_4, \\ (k_1+k_3,k_2+1,k_3,-k_4+2k_3) & \mbox{if }2\nmid k_4\end{cases}.$$ From this, it follows immediately that $\alpha$ has trivial kernel and hence is an automorphism of $G$. Also, note that the parity of the fourth component is invariant under $\alpha$. This makes it easy to determine the cycle length under $\alpha$ of an arbitrary group element $(k_1,k_2,k_3,k_4)$ in a case distinction:

If $k_4$ is even, then by an easy induction on $n$, for all $n\in\mathbb{N}$, we have that $$\alpha^n(k_1,k_2,k_3,k_4)=(k_1+n\cdot k_3,k_2,k_3,k_4+2n\cdot k_3).$$ In the subcase $k_3=0$, this formula simplifies to $$\alpha^n(k_1,k_2,0,k_4)=(k_1,k_2,0,k_4),$$ so the $5\cdot 3\cdot 2=30$ points of that form are fixed points under $\alpha$. And if $k_3=1$, the formula becomes $$\alpha^n(k_1,k_2,1,k_4)=(k_1+n,k_2,1,k_4+2n),$$ so apparently, all these $30$ points have cycle length $10$ under $\alpha$.

On the other hand, if $k_4$ is odd, we find that $$\alpha^n(k_1,k_2,k_3,k_4)=(k_1+n\cdot k_3,k_2+n,k_3,(-1)^nk_4+2n\cdot k_3).$$ For $k_3=0$, we thus have $$\alpha^n(k_1,k_2,0,k_4)=(k_1,k_2+n,0,(-1)^nk_4),$$ giving us $30$ points of cycle length $6$ under $\alpha$, and for $k_3=1$, the formula turns into $$\alpha^n(k_1,k_2,1,k_4)=(k_1+n,k_2+n,1,(-1)^nk_4+2n),$$ which yields $30$ points of cycle length $15$, and we conclude that $\alpha$ does not satisfy the RCC.
\end{proof}

Combining what we now know, we obtain:

\begin{theorem}\label{counterExTheo2}
There exist infinitely many finite non-RCC groups, the smallest of which have order $120$.\qed
\end{theorem}

\begin{corollary}\label{extCor}
The class of RCC-groups is not closed under extensions.
\end{corollary}

\begin{proof}
If it was, then by Corollary \ref{pCor1}, all finite solvable groups would be RCC-groups. However, by Theorem \ref{counterExTheo2}, there exists a non-RCC group of order $120$, which by Theorem \ref{counterExTheo} is solvable.
\end{proof}

\begin{corollary}\label{nearRingCor}
There exist infinitely many finite non-RCC (right) nearrings.
\end{corollary}

\begin{proof}
This follows immediately from Theorem \ref{counterExTheo2} by observing that for any group $G$, if we add, to the group structure of $G$, the trivial nearring multiplication $g\cdot h:=g$, then the automorphisms of the corresponding nearring structure on $G$ are just the automorphisms of the underlying group structure.
\end{proof}

\section{Subgroups of RCC- and non-RCC-groups}

We conclude by showing that all finite groups occur as subgroups of RCC- and of finite non-RCC-groups. The second statement is immediate by what we have shown so far:

\begin{proposition}\label{nonOccEmbedProp}
Any finite group is a direct factor of (and so in particular embeds into) some finite non-RCC-group.
\end{proposition}

\begin{proof}
For any finite group $G$, by Lemma \ref{splitLem}(1) and Proposition \ref{counterExProp}, the direct product $\mathrm{SmallGroup}(120,8)\times G$ is a non-RCC group.
\end{proof}

The first statement, in turn, follows immediately from the following:

\begin{proposition}\label{occEmbedProp}
Let $n\in\mathbb{N}$. The symmetric group $\mathcal{S}_n$ satisfies the RCC.
\end{proposition}

\begin{proof}
For $n\leq 4$, this is clear by Theorem \ref{counterExTheo}, so we may assume that $n\geq 5$. In this case, the result follows immediately from \cite[Theorem 1]{Hor74a}, since $\mathcal{S}_n$ has no nontrivial normal solvable subgroups.
\end{proof}

\section{Acknowledgements}

The author would like to thank Peter Hellekalek for his helpful comments and Michael Giudici for pointing out the paper \cite{GPS15a} to him.

\end{document}